\documentclass[11pt]{article}
\usepackage{cite,epic,eepic,euscript,verbatim,amsmath,amssymb,amsthm,afterpage,float,bm,stmaryrd,color}
\usepackage[normalem]{ulem}
\usepackage{graphicx,epsfig}
\usepackage{subfigure}
\usepackage{comment}

\usepackage{abstract}
\usepackage{graphicx}
\usepackage{epstopdf}

\usepackage{algorithm}
\usepackage{algpseudocode}

\usepackage{appendix}

\newtheorem{theorem}{Theorem}[section]
\newtheorem{lemma}{Lemma}[section]

\numberwithin{figure}{section}
\numberwithin{table}{section}

\def\XXint#1#2#3{{\setbox0=\hbox{$#1{#2#3}{\int}$}
\vcenter{\hbox{$#2#3$}}\kern-.51\wd0}}

\newcommand{\reff}[1]{{\rm (\ref{#1})}}
\newcommand{\R}{\mathbb{R}}            

\newcommand{\ve}{\varepsilon}          

{\allowdisplaybreaks}

\begin{document}

\title{A Positive and Energy Stable Numerical Scheme for the Poisson--Nernst--Planck--Cahn--Hilliard Equations with Steric Interactions}

\author{Yiran Qian\thanks{
Department of Mathematics and Mathematical Center for Interdiscipline Research, Soochow University, 1 Shizi Street, Suzhou 215006, Jiangsu, China.  E-mail: 20194007010@stu.suda.edu.cn.}
\and
Cheng Wang\thanks{
Department of Mathematics, University of Massachusetts Dartmouth, North Dartmouth, MA 02747-2300. E-mail: cwang1@umassd.edu.}
\and
Shenggao Zhou\thanks{
Corresponding author. Department of Mathematics and Mathematical Center for Interdiscipline Research, Soochow University, 1 Shizi Street, Suzhou 215006, Jiangsu, China. E-mail: sgzhou@suda.edu.cn.}
}
\maketitle
\begin{abstract}

In this work, we consider numerical methods for the Poisson--Nernst--Planck--Cahn--Hilliard (PNPCH) equations with steric interactions, which correspond to an $H^{-1}$ gradient flow of a free-energy functional that consists of electrostatic free energies, steric interaction energies of short range, entropic contribution of ions, and concentration gradient energies. We propose a novel energy stable numerical scheme that respects mass conservation and positivity at the discrete level.  Existence and uniqueness of the solution to the proposed nonlinear scheme are established by showing that the solution is a unique minimizer of a convex functional over a closed, convex domain. The positivity of numerical solutions is further theoretically justified by the singularity of the entropy terms, which prevents the minimizer from approaching zero concentrations.  A further numerical analysis proves discrete free-energy dissipation. Extensive numerical tests are performed to validate that the numerical scheme is first-order accurate in time and second-order accurate in space, and is capable of preserving the desired properties, such as mass conservation, positivity, and free energy dissipation, at the discrete level. Moreover, the PNPCH equations and the proposed scheme are applied to study charge dynamics and self-assembled nanopatterns in highly concentrated electrolytes that are widely used in electrochemical energy devices. Numerical results demonstrate that the PNPCH equations and our numerical scheme are able to capture nanostructures, such as lamellar patterns and labyrinthine patterns in electric double layers and the bulk, and multiple time relaxation with multiple time scales. The multiple time relaxation dynamics with metastability take long time to reach an equilibrium, highlighting the need for robust, energy stable numerical schemes that allow large time stepping.  In addition, we numerically characterize the interplay between cross steric interactions of short range and the concentration gradient regularization, and their impact on the development of nanostructures in the equilibrium state.  \\

\end{abstract}

\medskip

\bigskip

\section{Introduction}
Ion transport is fundamental to a wide variety of biophysical processes and technological applications, e.g., transmembrane ion channels, electrochemical energy devices, and electrokinetics in microfluidics\cite{IonChanel_HandbookCRC15, BazantDiffChg_PRE04, BazantReview_ACIS09}. Based on a mean-field approximation, the classical Poisson--Nernst--Planck (PNP) theory has been derived to describe ion dynamics in various scenarios. The diffusion and convection of ions under gradients of the electrostatic potential are modeled by the Nernst--Planck equations. In turn, the  electrostatic potential is governed by the Poisson equation with charge density arising from mobile ions.  Despite its success in many applications, the PNP theory is valid only for dilute solutions due to various underlying assumptions made in mean-field approximations~\cite{BazantReview_ACIS09, LiuJiXu_SIAP18}. For instance, it neglects ionic steric effects that play a crucial role in the description of concentrated electrolytes in confined environments, e.g., high ionic concentrations in ion channels.   

To address this issue, several versions of modified PNP theories with steric effects have been developed in the past few decades. One approach to account for steric effects is via the incorporation of entropy of solvent molecules to the electrostatic free energy~\cite{BAO_PRL97, BazantSteric_PRE07, Li_Nonlinearity09, ZhouWangLi_PRE11, BZLu_BiophyJ11, LiLiuXuZhou_Nonliearity13, NuoZhouMcCammon_JPCB14, BZLu_JCP14, JiZhou_CMS19}. One salient feature of this type of models is a  saturation concentration for compactly packed counterions in the vicinity of charge surfaces. Another strategy to include steric effects is to add an excess chemical potential, which can be given by the density functional theory\cite{JZWu_JPCM14, GLin_Cicp14}, or by the Lennard-Jones potential for hard-sphere repulsions~\cite{HyonLiuBob_CMS10, BobHyonLiu_JCP10, LinBob_CMS14}. These models often give rise to integro-differential equations that are computationally intractable. To avoid nonlocal integral terms, local approximations of nonlocal integrals up to leading order terms are proposed to obtain approximate local models~\cite{HyonLiuBob_CMS10, HorngLinLiuBob_JPCB12, LinBob_CMS14}. Nonetheless, the resulting model has been shown to be ill-posed for concentrated electrolytes within certain parameter regimes~\cite{LinBob_CMS14, Gavish_PhysD18}. To remedy this issue, concentration gradient energies that are higher order terms of the local approximations of nonlocal integrals can be added to regularize the solution~\cite{GavishYochelis_JPCLett16, GavishEladYochelis_JPCLett18, Gavish_PhysD18, GavishLiuEisenberg_JPCB18}. Such gradient energy terms are widely used in the Ginzburg--Landau theory for the description of phase separation in mixtures. A conserved $H^{-1}$ gradient flow of the Ginzburg-Landau functional gives rise to the Cahn--Hilliard (CH) equations~\cite{CH_JCP58}. 

An $H^{-1}$ gradient flow of the electrostatic free energy with additional concentration gradient energies leads to the  following Poisson--Nernst--Planck--Cahn--Hilliard (PNPCH) equations with steric interactions:
\begin{equation*}\left\{
\begin{aligned}
&\frac{\partial c^{m}}{\partial t}=\epsilon^{m}\nabla\cdot\left[c^{m}\nabla
\left(z^{m}\psi+\log c^{m}+\sum^{M}_{n=1}g^{mn}c^{n}-\sigma^{m}\Delta c^{m}\right)\right],\\
&-\nabla\cdot(\kappa\nabla\psi)=\sum^{M}_{m=1}z^{m}c^{m}+\rho^{f},
\end{aligned}\right.
\end{equation*}
where $\psi$ is the electrostatic potential, $c^{m}$ is the ion concentration for the $m$th species, $z^{m}$ is the valence, $\rho^{f}$ is the fixed charge density, $\kappa$ and $\epsilon^{m}$ arise from nondimensionalization, $G=\left(g^{mn}\right)$ is the coefficient matrix for steric interactions, and $\sigma^{m}$ is a gradient energy coefficient; cf.~Section~\ref{s:PNPCH}.  Such type of equations have been successfully applied to study ion permeation and selectivity in ion channels~\cite{GavishLiuEisenberg_JPCB18} and charge dynamics in room temperature ionic liquids and highly concentrated electrolytes~\cite{GavishYochelis_JPCLett16, GavishEladYochelis_JPCLett18}. 

We focus on the development of numerical methods for the PNPCH equations. Many efforts have been devoted to the development of numerical methods for the  PNP-type equations, ranging from finite difference schemes to discontinuous Galerkin (DG) methods~\cite{ProhlSchmuck09, LHMZ10, ZCW11, Gibou_JCP14, AFJKXL17, GaoHe_JSC17, DSWZhou_CICP18, DingWangZhou_NMTMA19}. In order to obtain physically faithful numerical solutions, it is highly desirable and crucial to preserve physical properties of the analytical solutions, such as mass conservation, free-energy dissipation, and positivity, at the discrete level.  A finite difference scheme was developed for the PNP equations in $1$D~\cite{AMEKLL14}; it was proved that the scheme guarantees numerical positivity if a time step size satisfies certain constraint conditions. An energy satisfying finite difference scheme based on a Slotboom transformation is proved to maintain discrete positivity under a constraint on the mesh ratio~\cite{LW14}.  An arbitrary-order free energy satisfying DG method is proposed to numerically solve the 1D PNP equations. The positivity of numerical solutions was enforced by a delicately designed accuracy-preserving limiter~\cite{LW17}.  A finite element method that can ensure positivity of numerical solutions was developed to solve both the PNP equations and PNP equations coupled with the Navier--Stokes equations~\cite{MXL16}. A semi-implicit finite difference scheme that ensures positivity and discrete energy dissipating properties was established in~\cite{HuHuang_Sub2019}. Based on harmonic-mean approximations~\cite{QianWangZhou_JCP19},  a finite difference scheme that is proved to respect mass conservation and unconditional positivity preservation was proposed for PNP equations with steric effects~\cite{DingWangZhou_JCP19}. Estimates on the condition number of the coefficient matrix was established as well. To the best of our knowledge, numerical methods with the desired properties for the PNPCH equations are still missing. 

In this work, we first derive the PNPCH equations corresponding to an $H^{-1}$ gradient flow of a free-energy functional that includes electrostatic free energies, steric interaction energies of short range, entropic contribution of ions, and concentration gradient energies. To numerically solve the PNPCH equations, we propose a novel energy stable numerical scheme that respects mass conservation and positivity at the discrete level. It is shown that the solution to the proposed nonlinear scheme corresponds to a unique minimizer of a convex functional over a closed, convex domain, establishing the existence and uniqueness of the solution. The positivity of numerical solutions is further theoretically justified by making use of the singular nature of the entropy terms, which prevents the minimizer from approaching zero concentrations. It is noted that such an argument has been used to prove the positivity preservation of numerical methods for the Cahn--Hilliard equations and quantum diffusion equations~\cite{CWWW_JCP2019, dong19b, HuoLiu_Sub20}; also see the related analysis~\cite{duan19a, duan19b} to deal with energetic variational method in the particle coordinate approach. Further numerical analysis establishes discrete free-energy dissipation of the proposed scheme. 

We perform extensive numerical tests to demonstrate that the numerical scheme is first-order accurate in time and second-order accurate in space, and is capable of preserving the desired properties, including mass conservation, positivity, and free energy dissipation. Moreover, the PNPCH equations and the proposed scheme are applied to study multiple time scale dynamics and self-assembled nanopatterns in highly concentrated electrolytes. The multiple time scale dynamics often take long time to reach an equilibrium, highlighting the need for robust, energy stable numerical schemes that allow large time stepping. Numerical results demonstrate that the PNPCH equations and the proposed numerical scheme are able to capture nanostructures, such as lamellar patterns and labyrinthine patterns, and multiple time scale dynamics with multiple time scales. In addition, we investigate the interplay between cross steric interactions of short range and the concentration gradient regularization, and their impact on the development of nanostructures in the equilibrium state.   


The rest of the paper is organized as follows. In Section~\ref{s:PNPCH} we derive the PNPCH equations from a free energy functional. In Section~\ref{s:NumericalScheme} we present the finite difference scheme. In Section~\ref{s:Properties} we prove main properties of the numerical scheme at the discrete level. Section~\ref{s:Numerics} is devoted to the numerical results. 
Finally, some concluding remarks are made in Section~\ref{s:Conclusions}.

\section{The physical model}\label{s:PNPCH}
We consider an electrolyte solution with $M $ ionic species, occupying a region $\Omega$ in $\mathbb{R}^{3}$. We denote by $c^{m}=c^{m}(t, \mathbf{x})$ $(m=1, \cdots, M)$ the local ionic concentration of the $m$th ionic species at position $\mathbf{x}\in\Omega$ and time $t$. 
Denote by $c=(c^{1}, c^{2},..., c^{M})^{T}$. For such a charged system, we consider the following free-energy functional of ionic concentrations:
\begin{equation}\label{CHFE}
\begin{aligned}
F[c]=\int_{\Omega}\frac{1}{2}\rho\psi d\textbf{x}+\beta^{-1}\sum^{M}_{m=1}\int_{\Omega}c^{m}\left[\log(\Lambda^{3}c^{m})-1\right] d\textbf{x}
+\int_{\Omega}\frac{1}{2}c^{T}Gc d\textbf{x}+\sum^{M}_{m=1}\int_{\Omega}\frac{\sigma^{m}}{2}|\nabla c^{m}|^{2} d\textbf{x}.
\end{aligned}
\end{equation}
The first term represents the electrostatic energy.  The total charge density $\rho$ is given by
$$\rho=\sum^{M}_{m=1}q^{m}c^{m}+\rho^{f},$$
where $q^{m}=z^{m}e$, with $z^{m}$ being the valence of the $m$th ionic species and $e$ being the elementary charge, and the function $\rho^f \in  C(\bar{\Omega}) $ represents the distribution of fixed charges. The electrostatic potential $\psi$ is governed by the Poisson's equation
\begin{equation}\label{POIeq}
\begin{aligned}
-\nabla\cdot(\varepsilon_{0}\varepsilon_{r}\nabla\psi)=\rho \quad \mathrm{in}\quad\Omega,
\end{aligned}
\end{equation}
where $\varepsilon_{0}$ is the vacuum dielectric permittivity and $\varepsilon_{r}\geq1$ is a spatially dependent 
dielectric coefficient function.  

The second term describes the entropic contribution of ions to the total free energy. The parameter $\beta=1/k_{B}T$ is the inverse of thermal energy, with $k_{B}$ being the Boltzmann constant and $T$ being the absolute temperature. The constant $\Lambda$ is the thermal de Broglie wavelength. The ionic steric interaction energy is given in the third term, in which the $M\times{M}$ matrix $G=(g^{mn})$ is symmetric with non-negative entries. The entry $g^{mn}$ is  related to the second-order virial coefficients of hard spheres, depending on the size of the $m$th and $n$th ionic species~\cite{ZhouJiangDoi_PRL17, DingWangZhou_JCP19}. Diagonal entries of $G$ describe self-steric interactions of ions of the same species, and off-diagonal entries correspond to short-range cross steric interactions between ions of different species. A similar model based on the leading order local approximation of the Lennard-Jones interaction energies has been developed in the works~\cite{HyonLiuBob_CMS10, HorngLinLiuBob_JPCB12, LinBob_CMS14, LinBob_Nonlinearity15, Gavish_PhysD18, SWZ_CMS18, GavishLiuEisenberg_JPCB18}, while the interaction matrix $G$ in that model has a different interpretation.

The fourth term, arising from the Cahn--Hilliard mixture theory~\cite{CH_JCP58}, describes a gradient energy that penalizes large concentration gradients. Here $\sigma^{m}$ are gradient energy coefficients. Since concentration gradient terms can be regarded as high-order approximation of the Lennard-Jones interaction energies, they have been recently introduced to the PNP theory to describe steric interactions in ionic liquids~\cite{GavishYochelis_JPCLett16, BGUYochelis_PRE17, GavishEladYochelis_JPCLett18} and ion channels~\cite{Gavish_PhysD18, GavishLiuEisenberg_JPCB18}. 


We shall derive governing equations based on the free-energy functional~\reff{CHFE}.
Taking the first variation of $F[c]$ with respect to ion concentrations, we obtain the chemical potential $\mu^{m}$ of the $m$th species of ions:
\begin{equation}\label{chemicalMU}
\begin{aligned}
\mu^{m}=\frac{\delta F[c]}{\delta c^{m}}=q^{m}\psi+\beta^{-1}\log(\Lambda^{3}c^{m})+\sum^{M}_{n=1}g^{mn}c^{n}-\sigma^{m}\Delta c^{m}.
\end{aligned}
\end{equation}
The force balance between thermodynamic force and hydrodynamic drag gives rise to the velocity of the $m$th ion species
$$\eta^{m}=-\beta D^{m}\nabla \mu^{m},$$
where $D^{m}$ is the diffusion constant.
The time evolution of $c^{m}$ satisfies the conservation equation
$$\frac{\partial c^{m}}{\partial t}=-\nabla \cdot (c^{m}\eta^{m}).$$
With the chemical potential~\reff{chemicalMU}, we obtain
\[
\begin{aligned}
&\frac{\partial c^{m}}{\partial t}=D^{m}\nabla\cdot\left\{\beta c^{m}\nabla\left[q^{m}\psi+\beta^{-1}\log(\Lambda^{3}c^{m})+\sum^{M}_{n=1}g^{mn}c^{n}-\sigma^{m}\Delta c^{m}\right]\right\}.
\end{aligned}
\]
Meanwhile, the following nondimensionalized variables are introduced~\cite{BazantSteric_PRE07, BazantDiffChg_PRE04}
\[
\begin{aligned}
&\tilde{\psi}=e\beta \psi, \quad \tilde{x}=\frac{x}{L}, \quad \tilde{c}^{m}=\frac{c^{m}}{c_{0}}, \quad \tilde{t}=\frac{tD_{0}}{L\lambda_{D}},  \quad \tilde{D}^{m}=\frac{D^{m}}{D_{0}}, \\
 &\tilde{v}=\Lambda^{3}c_{0},\quad  \tilde{\sigma}^{m}=\frac{\sigma^{m} c_{0}\beta}{L^{2}}, \quad \tilde{G}=c_{0}\beta G,
 \quad \tilde{\rho}^{f}=\frac{\rho^{f}}{ec_{0}},
 \end{aligned}
\]
where $c_{0}$ is a characteristic concentration, $L$ is a macroscopic length scale, and $\lambda_{D}$ is the Debye length given by
\[\begin{aligned}
\lambda_{D}:=\sqrt{\frac{\varepsilon_{0}\varepsilon_{r}}{2e^{2}c_{0}\beta}}.
\end{aligned}
\]
For simplicity, we omit the tildes and obtain the dimensionless Poisson--Nernst--Planck--Cahn--Hilliard (PNPCH) equations
\begin{equation}\label{PNPCHE}\left\{
\begin{aligned}
&\frac{\partial c^{m}}{\partial t}=\epsilon^{m}\nabla\cdot\left[c^{m}\nabla
\left(z^{m}\psi+\log c^{m}+\sum^{M}_{n=1}g^{mn}c^{n}-\sigma^{m}\Delta c^{m}\right)\right],\\
&-\nabla\cdot(\kappa\nabla\psi)=\sum^{M}_{m=1}z^{m}c^{m}+\rho^{f},
\end{aligned}\right.
\end{equation}
where $\epsilon^{m}=\frac{\lambda_{D}}{L}D^{m}$ and $\kappa=\frac{2\lambda_{D}^{2}}{L^{2}}$. Since the dielectric coefficient function $\varepsilon_{r}\geq 1 $, there exists a lower bound $\kappa_0$ of $\kappa$, i.e., 
\begin{equation}\label{kappa_0}
\kappa\geq \kappa_0:=\frac{\varepsilon_{0}}{e^{2}c_{0}L^2\beta}.
\end{equation}
After omitting the tildes over the dimensionless variables again, the free energy functional becomes
\begin{equation}\label{NONF}
\begin{aligned}
F[c]=&\int_{\Omega}\frac{1}{2}\left(\sum^{M}_{m=1}z^{m}c^{m}+\rho^{f}\right)\psi d\textbf{x}+\sum^{M}_{m=1}\int_{\Omega}c^{m}\left[\log(vc^{m})-1\right] d\textbf{x}\\
&\quad+\int_{\Omega}\frac{1}{2}c^{T}Gc d\textbf{x}+\sum^{M}_{m=1}\int_{\Omega}\frac{\sigma^{m}}{2}\left|\nabla c^{m}\right|^{2} d\textbf{x}.
\end{aligned}
\end{equation}

For simplicity, we assume that the domain $\Omega$ is a cuboid and consider periodic boundary conditions on the boundary $\partial \Omega$. Since $c^m$ represents the concentration of ions, it is reasonable to assume that  $c^{m}(t,\textbf{x}) \geq 0$ for $\textbf{x}\in \Omega$ and $t>0$. By periodic boundary conditions and equation system~\reff{PNPCHE},  we have mass conservation of each species of ions in the sense that
\[
\frac{d}{dt} \int_{\Omega} c^m(t, \textbf{x}) d\textbf{x} = \int_{\partial\Omega} c^{m} \epsilon^{m} \nabla
\left(z^{m}\psi+\log c^{m}+\sum^{M}_{n=1}g^{mn}c^{n}-\sigma^{m}\Delta c^{m}\right) \cdot  \textbf{n} dS =0,
\]
where $\textbf{n}$ is a unit normal vector defined on $\partial \Omega$.  Thus, the initial condition, 
$$c^m(0, \textbf{x}) = c^{m}_{in}(\textbf{x}),$$ determines the total mass of the $m$th species of ions in the system.  It is assumed that the initial data satisfy the neutrality condition
\begin{equation}\label{NeuCond}
\int_{\Omega} \rho^{f} d\textbf{x} + \sum^{M}_{m=1} \int_{\Omega} z^m c^{m}_{in}(\textbf{x}) d\textbf{x} =0,
\end{equation}
which is necessary for the solvability of the Poisson equation~\reff{POIeq} with periodic boundary conditions. 

We also consider time evolution of the free energy
\begin{align*}
\frac{d}{dt}F&= \sum^{M}_{m=1}\int_\Omega  \frac{\delta F}{\delta c^{m}}\frac{\partial c^{m}}{\partial t} d\textbf{x} \\
		  &= -\sum^{M}_{m=1}\int_\Omega \epsilon^{m}c^{m} \left| \nabla
\left(z^{m}\psi+\log c^{m}+\sum^{M}_{n=1}g^{mn}c^{n}-\sigma^{m}\Delta c^{m}\right)\right|^{2}d\textbf{x}\leq0,~\forall t>0. 
\end{align*}
In summary, we have the following physical properties for any solution to the PNPCH equations \reff{PNPCHE}:
\begin{align}
   &\bullet~\text{Positivity: } \quad \text{If }  c^{m}_{in}(\textbf{x})\geq0,~\text{then}~ c^{m}(t,\textbf{x})\geq0, ~\forall \textbf{x}\in \Omega ; \hspace{5cm}\label{3properties1} \\
  &\bullet~ \text{Mass Conservation: } \quad \int_{\Omega}c^{m}(t,\textbf{x})d\textbf{x}= \int_{\Omega}c^{m}_{in}(\textbf{x})d\textbf{x}; \label{3properties2}\\
 &\bullet~ \text{Free-energy Dissipation: } \quad \frac{d}{dt}F \leq0, ~\forall t>0. \label{3properties3}
\end{align}

\section{The numerical scheme}\label{s:NumericalScheme}
\subsection{Discretization preliminaries}
For simplicity, we present our numerical scheme in $\mathbb{R}^{3}$ with $\Omega=[a, b]\times[a, b]\times[a, b]$. We cover $\Omega$ with grid points 
\[
\left\{x_i, y_j,z_{k}\right\}= \left\{a+\left(i-\frac{1}{2}\right)h, a+\left(j-\frac{1}{2}\right)h, a+\left(k-\frac{1}{2}\right)h\right\} ~\text{for}~~ i,j,k=1,\cdots, N,
\]
where $N$ is the number of grid points in each dimension and $h=\frac{b-a}{N}$ is a uniform spatial mesh step size. 

We briefly recall notations and operators for discrete functions from~\cite{SCJ09,ChenLiuWangWise_MathComp15, CWWW_JCP2019}. To facilitate the presentation, we introduce the following spaces of $3D$ periodic grid functions:
\[
\begin{aligned}
&\mathcal{C}_{\mathrm{per}}:=\{v| v_{i,j,k}=v_{i+\alpha N,j+\beta N, k+\gamma N}, \quad\forall i,j,k, \alpha, \beta, \gamma  \in\mathbb{Z} \},\\
&\mathring{\mathcal{C}}_{\mathrm{per}}:=\{v \in C_{\mathrm{per}}|\bar{v} = 0\}, \\
&\mathcal{E}^{x}_{\mathrm{per}}:= \{v| v_{i+\frac{1}{2},j,k}=v_{i+\frac12 +\alpha N,j+\beta N, k+\gamma N}, \quad\forall i,j,k, \alpha, \beta, \gamma  \in\mathbb{Z} \},
\end{aligned}
\]
where $\bar{v}:=\frac{h^{3}}{|\Omega|}\sum^{N}_{i,j,k=1}v_{i,j,k}$ is the average of a grid function $v$. The spaces $\mathcal{E}^{y}_{\mathrm{per}}$ and $\mathcal{E}^{z}_{\mathrm{per}}$ are analogously defined. We also introduce the following discrete operators for grid functions:
\[
\begin{aligned}
&d_{x}v_{i,j,k}:=\frac{1}{h}(v_{i+\frac{1}{2},j,k}-v_{i-\frac{1}{2},j,k}), \quad D_{x}v_{i+\frac{1}{2},j,k}:=\frac{1}{h}(v_{i+1,j,k}-v_{i,j,k}),\\
&a_{x}v_{i,j,k}:=\frac{1}{2}(v_{i+\frac{1}{2},j,k}+v_{i-\frac{1}{2},j,k}),\quad A_{x}v_{i+\frac{1}{2},j,k}:=\frac{1}{2}(v_{i+1,j,k}+v_{i,j,k}) , 
\end{aligned}
\]
and the discrete operators $d_{y}$, $d_{z}$, $D_{y}$, $D_{z}$, $A_{y}$, and $A_{z}$ are similarly defined. The discrete gradient $\nabla_{h}$ becomes 
$$\nabla_{h}v_{i,j,k}:=(D_{x}v_{i+\frac{1}{2},j,k},D_{y}v_{i,j+\frac{1}{2},k},D_{z}v_{i,j,k+\frac{1}{2}}),$$
and the discrete divergence $\nabla_{h}\cdot$ turns out to be 
$$\nabla_{h}\cdot\vec{f}_{i,j,k}:=d_{x}f^{x}_{i,j,k}+d_{y}f^{y}_{i,j,k}+d_{z}f^{z}_{i,j,k}, \quad \mbox{for $\vec{f}=(f^{x},f^{y},f^{z})$.} 
$$
The discrete Laplacian operator $\Delta_{h}$ is defined by
\[
\begin{aligned}
\Delta_{h}v_{i,j,k}:&=\nabla_{h}\cdot(\nabla_{h}v)_{i,j,k}=d_{x}(D_{x}v)_{i,j,k}+d_{y}(D_{y}v)_{i,j,k}+d_{z}(D_{z}v)_{i,j,k}.
\end{aligned}
\]
If $\mathcal{D}$ is a periodic scalar function, we introduce 
\[
\begin{aligned}
\nabla_{h}\cdot(\mathcal{D} \nabla_{h}v)_{i,j,k}:=d_{x}(\mathcal{D} D_{x}v)_{i,j,k}+d_{y}(\mathcal{D} D_{y}v)_{i,j,k}+d_{z}(\mathcal{D} D_{z}v)_{i,j,k}.
\end{aligned}
\]


We now define the following inner product and norms for grid functions:  
\[
\begin{aligned}
&\langle\nu,\xi\rangle_{\Omega}:=h^{3}\sum^{N}_{i,j,k=1}\nu_{i,j,k}\xi_{i,j,k}, \quad \nu,\xi\in\mathcal{C}_{\mathrm{per}},\\
&[\nu,\xi]_{x}:=\langle a_{x}(\nu\xi),1\rangle_{\Omega}\quad \nu,\xi\in\mathcal{E}^{x}_{\mathrm{per}}.
\end{aligned}
\]
And also, $[\nu,\xi]_{y}$ and $[\nu,\xi]_{y}$ could be analogously defined. We then introduce 
\[
[\vec{f_{1}},\vec{f_{2}}]_{\Omega} :=[f^{x}_{1},f^{x}_{2}]_{x}+[f^{y}_{1},f^{y}_{2}]_{y}+[f^{z}_{1},f^{z}_{2}]_{z},\quad\vec{f_{i}}=(f^{x}_{i},f^{y}_{i},f^{z}_{i})\in\vec{\mathcal{E}}_{\mathrm{per}}, i=1,2.
\]

For $\nu\in\mathcal{C}_{\mathrm{per}}$, we define $\|\nu\|^{2}_{2}:=\langle\nu,\nu\rangle_{\Omega}$, $\|\nu\|^{p}_{p}:=\langle|\nu|^{p},1\rangle_{\Omega}$, for $1\leq p\leq\infty$, and $\|\nu\|_{\infty}:=\max_{1\leq i,j,k\leq N}|\nu_{i,j,k}|$.
For $\nu\in\ \mathcal{C}_{\mathrm{per}}$ and $1\leq p<\infty$, we define the following discrete norms of a gradient
\[
\begin{aligned}
\|\nabla_{h}\nu\|^{p}_{p}:&=\|D_{x}\nu\|^{p}_{p}+\|D_{y}\nu\|^{p}_{p}+\|D_{z}\nu\|^{p}_{p}.
\end{aligned}
\]
In addition, the higher order discrete norms are defined as 
\[
\begin{aligned}
\|\nu\|^{2}_{H^{1}_{h}}:=\|\nu\|^{2}_{2}+\|\nabla_{h}\nu\|^{2}_{2}, \quad \|\nu\|^{2}_{H^{2}_{h}}:=\|\nu\|^{2}_{2}+\|\nabla_{h}\nu\|^{2}_{2}+\|\Delta_{h}\nu\|^{2}_{2}.
\end{aligned}
\]

We now introduce a discrete analogue of the space $H^{-1}_{\mathrm{per}}(\Omega)$~\cite{WangWise2011, CWWW_JCP2019}.  Consider a positive, periodic scalar function $\mathcal{D}$. For any $\phi\in\mathring{\mathcal{C}}_{\mathrm{per}}$, there exists a unique solution $\varphi\in\mathring{\mathcal{C}}_{\mathrm{per}}$ to the equation
$$\mathcal{L}_{\mathcal{D}}\varphi:=-\nabla_{h}\cdot (\mathcal{D}\nabla_{h}\varphi)=\phi,$$ with periodic boundary conditions discretized as 
\begin{equation}\label{PBCs}
\begin{aligned}
\varphi_{N+1,j,k}&=\varphi_{1,j,k}, \quad \varphi_{0,j,k}=\varphi_{N,j,k}, \quad j,k=0, \cdots, N+1,\\
\varphi_{i,N+1,k}&=\varphi_{i,1,k}, \quad \varphi_{i,0,k}=\varphi_{i,N,k}, \quad i,k=0, \cdots, N+1,\\
\varphi_{i,j,N+1}&=\varphi_{i,j,1}, \quad \varphi_{i,j,0}=\varphi_{i,j,N}, \quad i,j=0, \cdots, N+1.
\end{aligned}
\end{equation}
For any $\phi_{1}, \phi_{2}\in \mathring{\mathcal{C}}_{\mathrm{per}},$ we define an inner product
$$\langle\phi_{1}, \phi_{2}\rangle_{\mathcal{L}^{-1}_{\mathcal{D}}}:=[\mathcal{D}\nabla_{h}\varphi_{1},\nabla_{h}\varphi_{2}]_{\Omega},$$
where $\varphi_{i}\in\mathring{\mathcal{C}}_{\mathrm{per}}$ is the unique solution to

\begin{equation}\label{LOp}
\begin{aligned}
\mathcal{L}_{\mathcal{D}}\varphi_{i}=\phi_{i}, \quad i=1,2.
\end{aligned}
\end{equation}
By the discrete summation by parts, formula we have the following identities for periodic grid functions $\phi_{i}$: 
$$\langle\phi_{1}, \phi_{2}\rangle_{\mathcal{L}^{-1}_{\mathcal{D}}}=\langle\phi_{1}, \mathcal{L}^{-1}_{\mathcal{D}}\phi_{2}\rangle_{\Omega}=\langle\mathcal{L}^{-1}_{\mathcal{D}}\phi_{1}, \phi_{2}\rangle_{\Omega}.$$
We denote the norm associated to this inner product by
$\|\phi\|_{\mathcal{L}^{-1}_{\mathcal{D}}}:=\sqrt{\langle\phi, \phi\rangle_{\mathcal{L}^{-1}_{\mathcal{D}}}}$. 

\subsection{The numerical scheme}
We consider the discretization of the PNPCH equations~\reff{PNPCHE} on a time interval $[0, T]$ with $T>0$.  For any function $v= v(x,y,z, t): \Omega \times [0, T] \rightarrow \R$, we denote by $v_{i,j,k}^l$ numerical approximations of $v(x_i, y_j, z_k, t_l)$ on a grid point $\left\{x_i, y_j, z_{k}\right\}$ at time $t_l=l \Delta t$, where $l=0, \cdots, N_t$ and $\Delta t = T/N_t$ with $N_t$ being a positive number.   

We employ the idea of convex splitting and propose a semi-implicit discrete scheme for the PNPCH equations~\reff{PNPCHE}:
\begin{equation}\label{MCS}\left\{
\begin{aligned}
&\frac{c^{m,l+1}-c^{m,l}}{\Delta{t}}=\epsilon^{m}\nabla_{h}\cdot\left(\check{c}^{m,l}\nabla_{h}\mu^{m,l+1}\right), \\ 
&\mu^{m,l+1}= z^{m}\psi^{l+1}+\log c^{m,l+1}+\sum^{M}_{n=1}g_{c}^{mn}c^{n,l+1}-\sigma^{m}\Delta_h c^{m,l+1}-\sum^{M}_{n=1}g_{e}^{mn}c^{n,l}, \\
&-\nabla_{h}\cdot (\kappa \nabla_{h}\psi^{l+1})=\sum^{M}_{m=1}z^{m}c^{m,l+1}+\rho_h^{f},
\end{aligned}\right.
\end{equation}
where the mobilities are approximated by
\[
\begin{aligned}
&\check{c}^{m,l}_{i+\frac{1}{2},j,k}=A_{x}c^{m,l}_{i+\frac{1}{2},j,k}, \quad \check{c}^{m,l}_{i,j+\frac{1}{2},k}=A_{y}c^{m,l}_{i,j+\frac{1}{2},k}, \quad \check{c}^{m,l}_{i,j,k+\frac{1}{2}}=A_{z}c^{m,l}_{i,j,k+\frac{1}{2}},
\end{aligned}
\]
and $\rho^{f}_h$, the restriction of $\rho^{f}$ on grid points, is assumed to satisfy
\begin{equation}\label{NeuInit}
\overline{\rho_h^f} + \sum_{m=1}^M z^m \overline{c^{m}_{in}} = 0.
\end{equation}
Here $G_{c}=\left(g^{mn}_{c}\right)$ and $G_{e}=\left(g^{mn}_{e}\right)$ are both positive semi-definite matrices such that $G=G_{c}-G_{e}$. Notice that the choice of $G_c$ and $G_e$ is not necessarily unique. In the numerical implementation, we choose the smallest positive $\lambda$ such that $G_{c}=\lambda I+G$ and $G_{e}=\lambda I$ are both positive semi-definite. 

The free-energy functional \reff{NONF} is discretized as 
\begin{equation}\label{SdisFE}
\begin{aligned}
F_{h}[c^l]=&\frac{1}{2}\left \|\sum^{M}_{m=1}z^{m}c^{m,l}+\rho_h^{f}\right \|^{2}_{\mathcal{L}^{-1}_{\kappa}}+\sum^{M}_{m=1}\left \langle c^{m,l},\log(vc^{m,l})-1 \right \rangle_{\Omega}\\
&\quad +\frac{1}{2}\sum^{M}_{n=1}\sum^{M}_{m=1}g^{mn}\langle c^{m,l},c^{n,l}\rangle_{\Omega}+\sum^{M}_{m=1}\frac{\sigma^{m}}{2}\left\|\nabla_{h}c^{m,l}\right\|^{2}_{2}.
\end{aligned}
\end{equation}

\section{The numerical properties}\label{s:Properties}
In this section, we show that the proposed numerical scheme~\reff{MCS} is uniquely solvable with desired properties in preserving positivity, mass conservation, and unconditional energy stability at the discrete level.

\subsection{Mass Conservation}
\begin{theorem}\label{t:MassCon}
The numerical concentration $c^{m,l}_{i,j,k}$ of the semi-implicit scheme~\reff{MCS} respects mass conservation, in the sense that the total concentration of each species remains constant in time, i.e.,
\begin{equation}\label{conserve2}
h^3\sum^{N}_{i,j,k=1}c^{m,l+1}_{i,j,k}=h^3\sum^{N}_{i,j,k=1}c^{m,l}_{i,j,k},  \quad l=0, \cdots, N_t-1.
\end{equation}
\end{theorem}
\begin{proof}
Summing both sides of the numerical scheme for concentrations over $i,j,k$ gives 
\begin{align*}
h^{3}\sum_{i,j,k=1}^{N}c_{i,j,k}^{m,l+1} -h^{3}\sum_{i,j,k=1}^{N}c_{i,j,k}^{m,l} = \Delta t \epsilon^{m} h^{3}\sum_{i,j,k=1}^{N} \nabla_{h}\cdot\left(\check{c}^{m,l}\nabla_{h}\mu_{i,j,k}^{m,l+1}\right)=0,
\end{align*}
where we have used the periodic boundary conditions~\reff{PBCs} and discrete summation by parts in the last step. This completes the proof of \reff{conserve2}.
\end{proof}
If these exists a solution to the numerical scheme~\reff{MCS}, it is indicated by Theorem~\ref{t:MassCon} that
\begin{equation}\label{AvgCon}
\overline{c^{m,0}}=  \overline{c^{m,1}} =\cdot\cdot\cdot =\overline{c^{m,N_t}}, \quad m =  1,2, \cdots, M.
\end{equation}
Therefore, the assumption~\reff{NeuInit} leads to a discrete version of the solvability condition: 
\[
\overline{\rho_h^f} + \sum_{m=1}^M z^m \overline{c^{m,l}} = 0, \quad  l = 0, 1, \cdots, N_t.
\]

\subsection{Positivity preservation}
To prove the positivity-preserving property of the proposed numerical scheme, we present the following lemma without giving its proof; cf.~\cite{CWWW_JCP2019}.

\begin{lemma}\label{PoiAssum}
Suppose that $\phi\in \mathring{\mathcal{C}}_{\mathrm{per}}$ and $\|\phi\|_{\infty}\leq C_1$, then we have the following estimate:
\begin{equation}\label{LEM}
\begin{aligned}
\|\mathcal{L}_{\mathcal{D}}^{-1}\phi\|_{\infty}\leq C_{3}:=C_{2}h^{-1/2}\mathcal{D}^{-1}_{0},
\end{aligned}
\end{equation}
where $\mathcal{D}_0$ is a positive lower bound of the coefficient function $\mathcal{D}(x)$, i.e.  $\mathcal{D}\geq \mathcal{D}_0 > 0$, and $C_2$ depends only on $C_1$ and $\Omega$.
\end{lemma}

\begin{theorem}\label{POSitive}
Assume that $c^{m,l}\in \mathcal{C}_{\mathrm{per}}$, and $c^{m,l}>0$ with $\|c^{m,l}\|_{\infty}\leq M_{1}$ for some $M_{1}>0$. Define $C^{m,l}_0:=\underset{1\leq i,j,k\leq N}{\operatorname{min}}c^{m,l}_{i,j,k}$, so that $c^{m,l} \geq C^{m,l}_0>0$. There exists a unique solution $c^{m,l+1}\in\mathcal{C}_{\mathrm{per}}$ to the nonlinear scheme \reff{MCS} with $c^{m,l+1}>0$ at a point-wise level.
\end{theorem}

\begin{proof}
The numerical solution of the nonlinear scheme \reff{MCS} corresponds to the minimizer of the discrete energy functional
\[
\begin{aligned}
\mathcal{J}^{k}(u)=&\frac{1}{2\triangle{t}}\sum^{M}_{m=1}\|u^{m}-c^{m,l}\|^{2}_{\mathcal{L}^{-1}_{\check{c}^{m,l}}}+
\frac{1}{2}\|\sum^{M}_{m=1}z^{m}u^{m}+\rho_h^{f}\|^{2}_{\mathcal{L}^{-1}_{\kappa}}
+\sum^{M}_{m=1}\left\langle u^{m},\log u^{m}-1\right\rangle_{\Omega}\\
&+\frac{1}{2}\sum^{M}_{n=1}\sum^{M}_{m=1}g^{mn}_{c}\langle u^{m}, u^{n}\rangle_{\Omega}+\sum^{M}_{m=1}\frac{\sigma^{m}}{2}\|\nabla_{h}u^{m}\|^{2}_{2}-\sum^{M}_{n=1}\sum^{M}_{m=1}g^{mn}_{e}\langle u^{m}, c^{n,k}\rangle_{\Omega},
\end{aligned}
\]
over the admissible set
$$A_{h}:=\{u^{m}\in \mathcal{C}_{\mathrm{per}}| 0 < u^{m} < \theta^m, ~ \overline{u^{m}}=\overline{c^{m,0}}~ \text{for}~ 1\leq m \leq M \}\subset\mathbb{R}^{MN^{3}},$$ where $\theta^m=\frac{|\Omega| \overline{c^{m,0}}}{h^3}$. 
It is easy to verify that $\mathcal{J}^{k}$ is strictly convex over this domain. We shall prove that the minimizer of $\mathcal{J}^{k}$ exists in $A_{h}$ and is positive on each grid point. We first consider a closed subset
$$A_{h, \delta}:=\{u^{m}\in \mathcal{C}_{\mathrm{per}}|\delta \leq u^{m} \leq  \theta^m -\delta, ~ \overline{u^{m}}=\overline{c^{m,0}}~ \text{for}~ 1\leq m \leq M\},$$
where $\delta$ is a number in $(0, \theta^m/2)$. 
Since $A_{h, \delta}$ is a bounded, compact, and convex subset in $\mathcal{C}_{\mathrm{per}}$, there exists a unique minimizer of $\mathcal{J}^{k}$ over $A_{h, \delta}$. Let the minimizer be $u^{\ast}=(u^{1, \ast}, u^{2, \ast}, \cdots, u^{M, \ast})$. When $\delta$ is sufficiently small, $u^{\ast}$ could not reach the lower boundary of $A_{h, \delta}$.  We shall prove this by contradiction. 

Suppose that there exists one grid point $\vec{\alpha}_{0}=(i_{0}, j_{0}, k_{0})$ and $m_0$ such that  $u^{m_0,\ast}$ achieves its global minimum at $\vec{\alpha}_{0}$ with $u^{m_0,\ast}_{\vec{\alpha}_{0}}=\delta$. Suppose that $\vec{\alpha}_{1}=(i_{1}, j_{1}, k_{1})$ is another grid point at which $u^{m_0,\ast}$ achieves its global maximum. Obviously, we have
$$\overline{c^{m_0,0}} \leq u^{m_0,\ast}_{\vec{\alpha}_{1}}\leq \theta^{m_0}-\delta.$$
Since $\mathcal{J}^{k}$ is smooth over $A_{h}$ and $u^{\ast}\in A_{h, \delta}$, the following directional derivative is well defined for sufficiently small $s$:
\[
\begin{aligned}
&\underset{s\to 0+}{\operatorname{\lim}}\frac{\mathcal{J}^{k}(u^{\ast}+s\phi)- \mathcal{J}^{k}(u^{\ast})}{s} \\
&\quad=\frac{1}{\triangle{t}}\left\langle\mathcal{L}^{-1}_{\check{c}^{m_0,l}}\left(u^{m_0,\ast}-c^{m_0,l}\right), \phi^{m_0}\right\rangle_{\Omega}+\left\langle z^{m_0}\mathcal{L}^{-1}_{\kappa}\left(\sum^{M}_{m=1}z^{m}u^{m,\ast}+\rho_h^{f}\right), \phi^{m_0}\right\rangle_{\Omega}\\
&\quad-\sigma^{m_0}\left\langle\Delta_{h}u^{m_0,\ast}, \phi^{m_0}\right\rangle_{\Omega}
+\langle \log u^{m_0,\ast}, \phi^{m_0}\rangle_{\Omega}+\sum^{M}_{n=1}g^{m_0n}_{c}\langle u^{n,\ast},\phi^{m_0}\rangle_{\Omega}-\sum^{M}_{n=1}g^{m_0n}_{e}\langle c^{n,l},\phi^{m_0}\rangle_{\Omega},
\end{aligned}
\]
where $\phi=(0, \cdots, \phi^{m_0}, \cdots, 0)$.
If we choose the direction
$$\phi_{i,j,k}^{m_0}=\delta_{i,i_{0}}\delta_{j,j_{0}}\delta_{k,k_{0}}-\delta_{i,i_{1}}\delta_{j,j_{1}}\delta_{k,k_{1}}\in \mathring{\mathcal{C}}_{\mathrm{per}},$$
then the directional  derivative becomes
\begin{equation}\label{DERVA}
\begin{aligned}
\frac{1}{h^{3}}\underset{s\to 0+}{\operatorname{\lim}}\frac{\mathcal{J}^{k}(u^{\ast}+s\phi)- \mathcal{J}^{k}(u^{\ast})}{s}&=\frac{1}{\triangle{t}}\mathcal{L}^{-1}_{\check{c}^{m_0,l}}(u^{m_0,\ast}-c^{m_0,l})_{\vec{\alpha}_{0}}
-\frac{1}{\triangle{t}}\mathcal{L}^{-1}_{\check{c}^{m_0,l}}(u^{m_0,\ast}-c^{m_0,l})_{\vec{\alpha}_{1}}\\
&+z^{m_0}\mathcal{L}^{-1}_{\kappa}\left(\sum^{M}_{m=1}z^{m}u^{m,\ast}+\rho_h^{f}\right)_{\vec{\alpha}_{0}}-z^{m_0}\mathcal{L}^{-1}_{\kappa}\left(\sum^{M}_{m=1}z^{m}u^{m,\ast}+\rho_h^{f}\right)_{\vec{\alpha}_{1}}\\
&-\sigma^{m_0}\left(\Delta_{h}u^{m_0,\ast}_{\vec{\alpha}_{0}}-\Delta_{h}u^{m_0,\ast}_{\vec{\alpha}_{1}}\right)
+\log u^{m_0,\ast}_{\vec{\alpha}_{0}}-\log u^{m_0,\ast}_{\vec{\alpha}_{1}}\\
&+\sum^{M}_{n=1}g^{m_0n}_{c}u^{n,\ast}_{\vec{\alpha}_{0}}-\sum^{M}_{n=1}g^{m_0n}_{c}u^{n,\ast}_{\vec{\alpha}_{1}}-\sum^{M}_{n=1}g^{m_0n}_{e}c^{n,l}_{\vec{\alpha}_{0}}+\sum^{M}_{n=1}g^{m_0n}_{e}c^{n,l}_{\vec{\alpha}_{1}}.
\end{aligned}
\end{equation}
Since $u^{m_0,\ast}_{\vec{\alpha}_{0}}=\delta$ and $u^{m_0,\ast}_{\vec{\alpha}_{1}}\geq\overline{c^{m,0}}$, we have
\begin{equation}\label{DER1}
\begin{aligned}
\log u^{m_0,\ast}_{\vec{\alpha}_{0}}-\log u^{m_0,\ast}_{\vec{\alpha}_{1}} \leq \log \delta-\log \overline{c^{m,0}}.
\end{aligned}
\end{equation}
Since $u^{m_0,\ast}$ takes a minimum at the grid point $\vec{\alpha}_{0}$ and a maximum at the grid point $\vec{\alpha}_{1}$, we obtain 
\begin{equation}\label{DER2}
\begin{aligned}
\Delta_{h}u^{m_0,\ast}_{\vec{\alpha}_{0}}\geq 0,\quad \Delta_{h}u^{m_0,\ast}_{\vec{\alpha}_{1}}\leq 0.
\end{aligned}
\end{equation}
It follows from $g^{mn}_{c}>0$, $g^{mn}_{e}>0$, $u^{n,\ast}>0$, and $c^{n,l}>0$ that
\begin{equation}\label{DER3}
-\sum^{M}_{n=1}g^{m_0n}_{c}u^{n,\ast}_{\vec{\alpha}_{1}}-\sum^{M}_{n=1}g^{m_0n}_{e}c^{n,l}_{\vec{\alpha}_{0}} <0.
\end{equation}
Since $u^{n,\ast}_{\vec{\alpha}_{0}}\leq \theta^n- \delta$, we have 
 \begin{equation}\label{DER4}
 \sum^{M}_{n=1}g^{m_0n}_{c}u^{n,\ast}_{\vec{\alpha}_{0}} \leq \sum^{M}_{n=1}g^{m_0n}_{c} \left(\theta^n- \delta\right)  \leq \sum^{M}_{n=1}g^{m_0n}_{c} \theta^n.
 \end{equation}
Also, the \emph{a priori} assumption $\|c^{n,l}\|_{\infty}\leq M_{1}$ indicates that
 \begin{equation}\label{DER5}
\begin{aligned}
\sum^{M}_{n=1}g^{m_0n}_{e}c^{n,l}_{\vec{\alpha}_{1}} \leq M_{1} \sum^{M}_{n=1}g^{m_0n}_{e}.
\end{aligned}
\end{equation}
Since $\check{c}^{m_0,l} \geq C_0^{m_0,l}>0$ and $\kappa \geq \kappa_0>0$ (cf.~\reff{kappa_0}), we have by applying the Lemma \ref{PoiAssum} with $\mathcal{D}=\check{c}^{m_0,l}$ and $\mathcal{D}=\kappa$ that 
\begin{equation}\label{DER6}
\begin{aligned}
 \mathcal{L}^{-1}_{\check{c}^{m_0,l}}(u^{m_0,\ast}-c^{m_0,l})_{\vec{\alpha}_{0}}
- \mathcal{L}^{-1}_{\check{c}^{m_0,l}}(u^{m_0,\ast}-c^{m_0,l})_{\vec{\alpha}_{1}}\leq 2C_{3}^c,
\end{aligned}
\end{equation}
and 
\begin{equation}\label{DER7}
\begin{aligned}
z^{m_0}\mathcal{L}^{-1}_{\kappa}\left(\sum^{M}_{m=1}z^{m}u^{m,\ast}+\rho_h^{f}\right)_{\vec{\alpha}_{0}}-z^{m_0}\mathcal{L}^{-1}_{\kappa}\left(\sum^{M}_{m=1}z^{m}u^{m,\ast}+\rho_h^{f}\right)_{\vec{\alpha}_{1}} \leq 2\left|z^{m_0}\right|C_{3}^{\kappa},
\end{aligned}
\end{equation}
respectively.  Note that we have used the assumption that $\check{c}^{m_0,l}\geq C^{m_0,l}_0>0$ and $\kappa\geq \kappa_0>0$; cf.~\reff{kappa_0}. The constant $C_{3}^c$ depends on $\theta^{m_0}$, $M_1$, $\Omega$, $h$, and $C^{m_0,l}_0$; and the constant $C_{3}^{\kappa}$ depends on $\text{max}_{1\leq m \leq M} \{|z^m| \theta^{m}\}$, $\|\rho_h^f\|_{\infty}$, $\Omega$, $h$, and $\kappa_0$.
Thus, a substitution of \reff{DER1}$-$\reff{DER7} into \reff{DERVA} leads to 
\begin{equation}\label{DERi}
\begin{aligned}
&\frac{1}{h^{3}}\underset{s\to 0+}{\operatorname{\lim}}\frac{\mathcal{J}^{k}(u^{\ast}+s\phi)- \mathcal{J}^{k}(u^{\ast})}{s} \\
&\qquad\leq  \log\delta-\log\overline{c^{m,0}}+ \sum^{M}_{n=1}g^{m_0n}_{c} \theta^n  + M_{1} \sum^{M}_{n=1}g^{m_0n}_{e}  +2C_{3}^{c}\triangle{t}^{-1}+2\left|z^{m_0}\right|C_{3}^{\kappa}. 
\end{aligned}
\end{equation}
For any fixed $\triangle{t}$ and $h$, we may choose $\delta$ sufficiently small so that
\begin{equation}\label{DER7}
\begin{aligned}
\log\delta-\log\overline{c^{m,0}}+ \sum^{M}_{n=1}g^{m_0n}_{c} \theta^n  + M_{1} \sum^{M}_{n=1}g^{m_0n}_{e}  +2C_{3}^{c}\triangle{t}^{-1}+2\left|z^{m_0}\right|C_{3}^{\kappa}<0.
\end{aligned}
\end{equation}
That is 
\begin{equation}\label{finaDER}
\begin{aligned}
\underset{s\to 0+}{\operatorname{\lim}}\frac{\mathcal{J}^{k}(u^{\ast}+s\phi)- \mathcal{J}^{k}(u^{\ast})}{s}<0.
\end{aligned}
\end{equation}
This is contradictory to the assumption that $u^{\ast}$ is the minimizer of $\mathcal{J}^{k}$, since the direction $\phi$ we chose points into the interior of $A_{h,\delta}$.  

Since the numerical solution of each species conserves at each time step,  one point of $u^{\ast}$ approaching $\theta^m -\delta$ implies many points of $u^{\ast}$ going to zero, when $\delta$ is sufficiently small. Thus, we can analogously show that the $u^{\ast}$ cannot reach the upper boundary of $A_{h,\delta}$. 

Therefore, when $\delta$ is sufficiently small, the global minimum of $\mathcal{J}^{k}$ over $A_{h,\delta}$ can only be achieved at an interior point of $A_{h,\delta}$, which is a subset of $A_{h}$. This establishes the existence of a positive numerical solution to the nonlinear scheme~\reff{MCS}.  In addition, the strict convexity of $\mathcal{J}^{k}$ over $A_{h}$ implies the uniqueness of the numerical solution. The proof of Theorem \ref{POSitive} is complete.
\end{proof}

\subsection{Unconditional energy stability}
\begin{theorem}
The semi-implicit discrete scheme~\reff{MCS} is energy stable, in the sense that
\[
\begin{aligned}
F_{h}(c^{l+1})\leq F_{h}(c^{l}).
\end{aligned}
\]
\end{theorem}

\begin{proof}
Since $\mathcal{L}^{-1}_{\kappa}$ is symmetric, positive definite on the space $\mathring{\mathcal{C}}_{\mathrm{per}}$, we know that the term $\frac{1}{2}\left\|\sum^{M}_{m=1}z^{m}c^{m}+\rho_h^{f}\right\|^{2}_{\mathcal{L}^{-1}_{\kappa}}$ is convex with respect to $c^m$. A direct calculation reveals that the term 
\[
\sum^{M}_{m=1}\left \langle c^{m},\log(vc^{m})-1\right\rangle_{\Omega}+\sum^{M}_{m=1}\frac{\sigma^{m}}{2}\left \|\nabla_{h}c^{m}\right \|^{2}_{2}
\]
is convex as well. We know  by the positive semi-definiteness of the matrices $G_c$ and $G_e$ that 
\[
  \frac{1}{2}\sum^{M}_{n=1}\sum^{M}_{m=1}g^{mn}_{c}\left \langle c^{m},c^{n}\right \rangle_{\Omega}~~\text{and }~ \frac{1}{2}\sum^{M}_{n=1}\sum^{M}_{m=1}g^{mn}_{e}\left \langle c^{m},c^{n}\right\rangle_{\Omega}  \, \, \, 
\mbox{are convex} . 
\]
Therefore, by the convexity of these terms and mass conservation~\reff{AvgCon}, we arrive at   
\[
\begin{aligned}
F_{h}(c^{l+1})-F_{h}(c^{l})
\leq&\sum^{M}_{m=1}\left \langle z^{m}\psi^{l+1}+\log c^{m,l+1}+\sum^{M}_{n=1}g^{mn}_{c}c^{n,l+1}-\sigma^{m}\Delta _{h}c^{m,l+1}\right.\\
&\left.\quad-\sum^{N}_{n=1}g^{mn}_{e} c^{n,l},c^{m,l+1}-c^{m,l}\right \rangle_{\Omega}\\
=&\sum^{M}_{m=1}\left \langle\mu^{m,l+1},\Delta{t}\epsilon^{m}\nabla_{h}\cdot\left(\check{c}^{m,l}\nabla_{h}\mu^{m,l+1}\right) \right \rangle_{\Omega}\\
=&- \sum^{M}_{m=1} \Delta{t} \epsilon^{m}  \left  [ \nabla_{h} \mu^{m,l+1},  \check{c}^{m,l}\nabla_{h}\mu^{m,l+1}  \right ]_{\Omega} \leq 0,
\end{aligned}
\]
in which the periodic boundary conditions~\reff{PBCs} and discrete summation by parts formulas have been used. This completes the proof.
\end{proof}

\section{Numerical examples}\label{s:Numerics}
At each time step evolution, we numerically solve the nonlinear difference equation~\reff{MCS} supplemented with periodic boundary conditions~\reff{PBCs} using the Newton's iterations. The Newton's iterations with ion concentrations, chemical potentials, and the electrostatic potential as unknowns converge robustly within four stages in our extensive numerical tests. For simplicity, we consider a periodic charged system consists of concentrated binary mononvalent electrolytes and fixed charges.   Unless stated otherwise, we take the characteristic concentration $c_{0}=1$ M, characteristic length $L=1$ nm,  characteristic diffusion constant $D_{0}=1$ nm$^{2}$/ns, and diffusion constants $D^{1}=D^{2}=1$ nm$^{2}$/ns for two species of ions. We consider a uniform dielectric constant $\ve_r=78$, which prescribes $\lambda_D = 0.304$ nm, $\kappa = 0.185$, and $\epsilon^1=\epsilon^2=0.304$.

\subsection{Accuracy test}
We test the numerical convergence order of the proposed numerical scheme~\reff{MCS} in one dimension. To obtain a reference solution for comparison, we construct an exact solution
\begin{equation}\label{exact1D}
\left\{
\begin{aligned}
&c^{1}=0.1e^{-t}\cos(\pi x)+0.2,\\
&c^{2}=0.1e^{-t}\cos(\pi x)+0.2,\\
&\psi=e^{-t}\cos(\pi x),
\end{aligned}\right.
\end{equation}
to the PNPCH equations
\begin{equation}\label{DCH}
\left\{
\begin{aligned}
&\frac{\partial c^{m}}{\partial t}=\epsilon^{m}\partial_x \left[c^{m}\partial_x
\left(z^{m}\psi+\log c^{m}+\sum^{M}_{n=1}g^{mn}c^{n}-\sigma^{m}\Delta c^{m}\right)\right] + f_m, \quad m =1, 2, \\
&- \partial_x (\kappa \partial_x \psi)=\sum^{2}_{m=1}z^{m}c^{m}+\rho^{f},
\end{aligned}\right.
\end{equation}
with periodic boundary conditions. 
Here the source terms $f_{1}$, $f_{2}$, and $\rho^{f}$, and the initial conditions are determined by the known exact solution~\reff{exact1D}. We choose the computational domain $\Omega=[-1, 1]$, and take the steric interaction coefficient matrix $
G=\left(
\begin{matrix}
3.6 & 2.6\\
2.6  & 0.2 \\
\end{matrix}
\right) $ and gradient energy coefficients $\sigma^{1}=\sigma^{2}=0.01.$ Note that the matrix $G$ is not positive semi-definite and therefore the corresponding free energy functional~\reff{NONF} is nonconvex. 

\begin{table}[htbp]
\centering
\begin{tabular}{ccccccc}
\hline \hline
 $N$ & $\ell^\infty$ error in $c^1$ & Order & $\ell^\infty$ error in $c^2$ & Order & $\ell^\infty$ error in $\psi$ & Order\\
 \hline
100 & 3.98e-5 & - & 3.94e-5 & - & 6.57e-4 & -\\
 200 & 9.97e-6 & 1.9963 & 9.87e-6 & 1.9969  & 1.64e-4 & 2.0003\\
 400 & 2.50e-6 & 1.9985 &  2.47e-6 & 1.9992  & 4.11e-5 & 2.0001\\
 800 & 6.24e-7 & 1.9993 &  6.17e-7 & 1.9998  & 1.03e-5 & 2.0001\\
 \hline \hline
\end{tabular}
\caption{The $\ell^\infty$ error and convergence order for numerical solutions of $c^1$, $c^2$, and $\psi$ at $T=0.0016$ with a mesh ratio $\Delta t=h^2$. }
\label{order1d}
\end{table}

We test numerical accuracy of the proposed scheme~\reff{MCS} on various meshes, in comparison with the exact solution~\reff{exact1D}. Notice that the mesh ratio, $\Delta t=h^{2}$, is chosen for the purpose of numerical accuracy test, rather than for the concern of numerical stability.    Table~\ref{order1d} lists the $\ell^\infty$ error and convergence order for the numerical solutions of ion concentrations and the electrostatic potential at time $T=0.0016$. It is observed that, the numerical error decreases as the mesh refines, and the convergence orders for ion concentrations and the potential are both around $2$, as expected. This confirms that the proposed numerical scheme~\reff{MCS} is second order accurate in space and first order accurate in time. 

%

%
We also test numerical accuracy of the proposed scheme~\reff{MCS}  in two dimensions. We take $\Omega=[-4, 4] \times [-4, 4]$ and the steric interaction coefficient matrix $
G=\left(
\begin{matrix}
2 & 1\\
1 & 2 \\
\end{matrix}
\right).$ 
The following exact solution is constructed 
\begin{equation}\label{exact2D}
\left\{
\begin{aligned}
&c^{1}=0.1e^{-20t}\cos(\pi x/4)\sin(\pi y/4)+1,\\
&c^{2}=0.1e^{-20t}\cos(\pi x/4)\sin(\pi y/4)+1,\\
&\psi=e^{-20t}\cos(\pi x)\sin(\pi y/4),
\end{aligned}\right.
\end{equation}
for the equations~\reff{DCH} with periodic boundary conditions. Again, the corresponding source terms and the initial conditions are determined by the known exact solution~\reff{exact2D}.  

\begin{table}[htbp]
\centering
\begin{tabular}{ccccccc}
\hline \hline
 $N$ & $\ell^\infty$ error in $c^1$ & Order & $\ell^\infty$ error in $c^2$ & Order & $\ell^\infty$ error in $\psi$ & Order\\
 \hline
20 & 3.39e-1 & - & 3.39e-1 & - & 1.24e-1 & -\\
40 & 8.38e-2 & 2.0158 & 8.38e-2 & 2.0158  & 2.78e-2 & 2.1549\\
60 & 3.70e-2 & 2.0162 & 3.70e-2 &  2.0162 & 1.21e-2 & 2.0515 \\
80 & 2.07e-2 & 2.0188 & 2.07e-2  &  2.0188 & 6.80e-3 & 2.0032\\
 \hline \hline
\end{tabular}
\caption{The $\ell^\infty$ error and convergence order for numerical solutions of $c^1$, $c^2$, and $\psi$ at $T=0.16$ with a mesh ratio $\Delta t=h^2$.}
\label{order2d}
\end{table}

Similarly, we carry out computations on various meshes with $\Delta t=h^2$ and compare with the exact solution~\reff{exact2D}. As shown in Table~\ref{order2d}, the numerical solutions of ion concentrations and the potential both converge to the exact solution with a convergence rate around $2$, indicating the anticipated accuracy order of the proposed numerical scheme~\reff{MCS} in the $2$D case. 

\subsection{Properties tests}
\begin{figure}[htbp]
\centering
\includegraphics[scale=0.52]{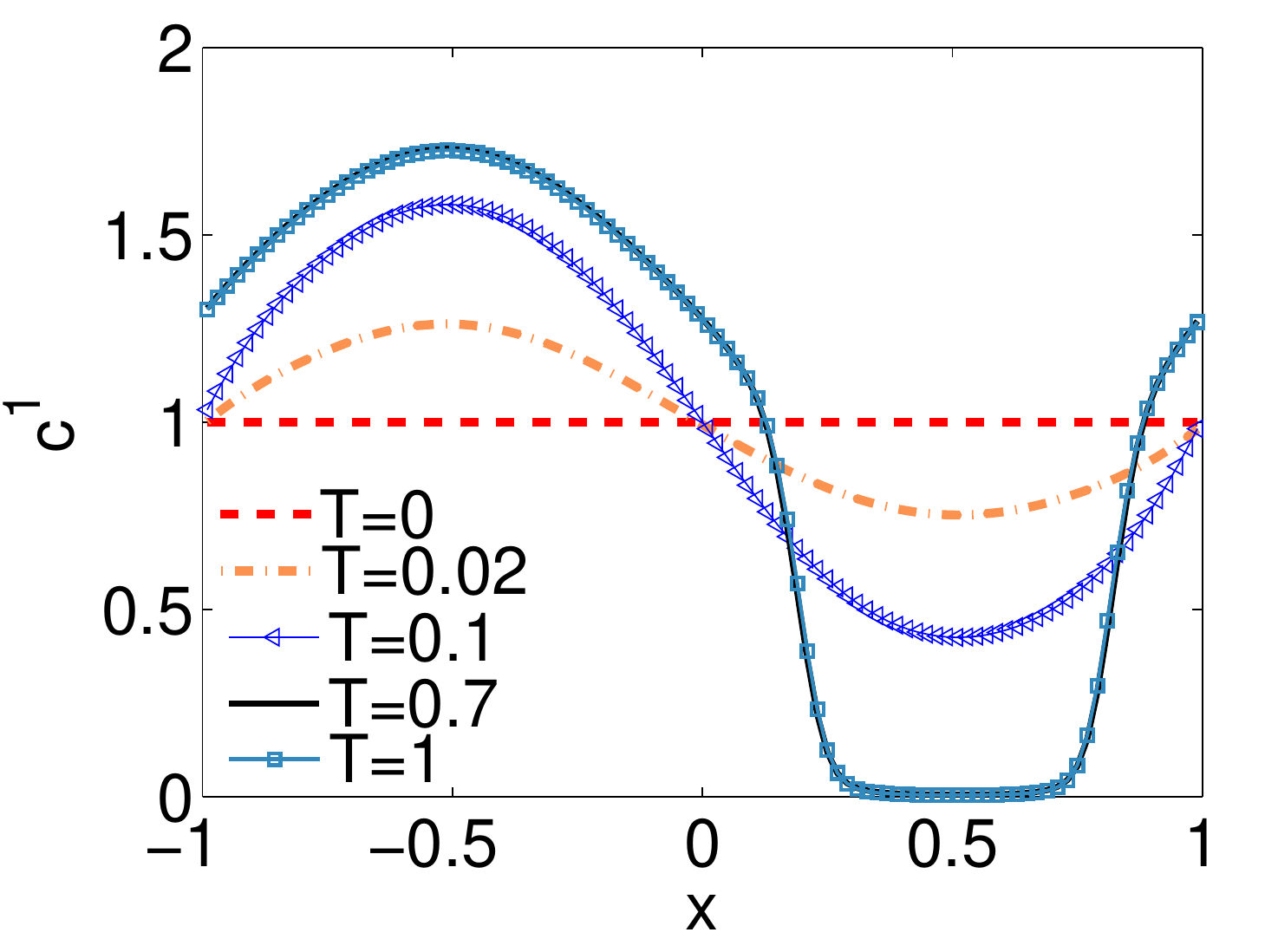}
\includegraphics[scale=0.52]{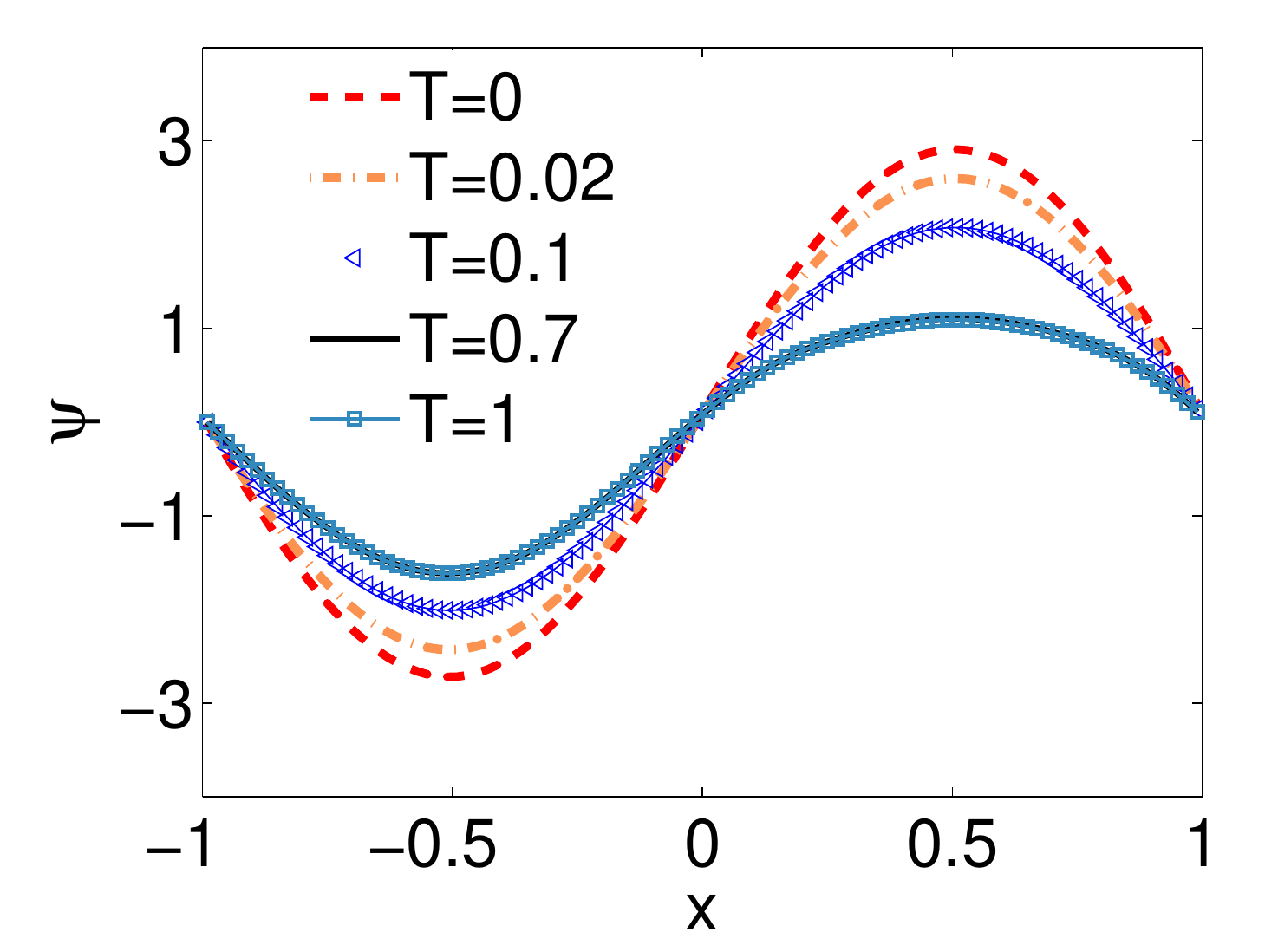}
\caption{Time evolution of numerical solutions of cation concentrations ($c^1$) and potential $\psi$.}
\label{c1c2psi}
\end{figure}

We now test the performance of the proposed scheme in preserving the desired properties, including positivity, mass conservation, and energy dissipation, at the discrete level. We consider the equations~\reff{PNPCHE}  on $\Omega=[-1,1]$ with periodic boundary conditions and initial data  
\[
c^1(x, 0)=c^2(x, 0)=1.
\]
The steric interaction coefficient matrix is taken as $
G=\left(
\begin{matrix}
3.6 & 2.6\\
2.6  & 0.2 \\
\end{matrix}
\right) $, and the charge distribution function is given by 
$$\rho^{f}(x)=5\left[e^{-5(x-\frac12)^2}-e^{-5(x+\frac12)^2}\right],$$ which describes that negative and positive fixed charges are distributed at $x=-\frac12$ and $x=\frac12$, respectively. 
In the numerical simulations, we set the total grid number $N=100$ and a mesh ratio $\Delta t = h$.

Figure~\ref{c1c2psi} displays the snapshot evolution of cation $c^{1}$ concentration and $\psi$ at different times. Since cations and anions distribute uniformly on $\Omega$ at $T=0$, the electrostatic potential $\psi$ is determined by the fixed charges, with maximum and minimum values at $x=\frac12$ and $x=-\frac12$, respectively. As time evolves, the  cations are attracted by the negative fixed charges and get repelled by positive fixed charges, leading to sinusoidal profiles. Accordingly, the electrostatic potential gets screened by ion accumulation. We observe that the profiles at $T=1$ are almost identical to that of $T=0.7$, which implies that the charges in the system have arrived at a steady state. 

\begin{figure}[H]
\centering
{\includegraphics[scale=0.7]{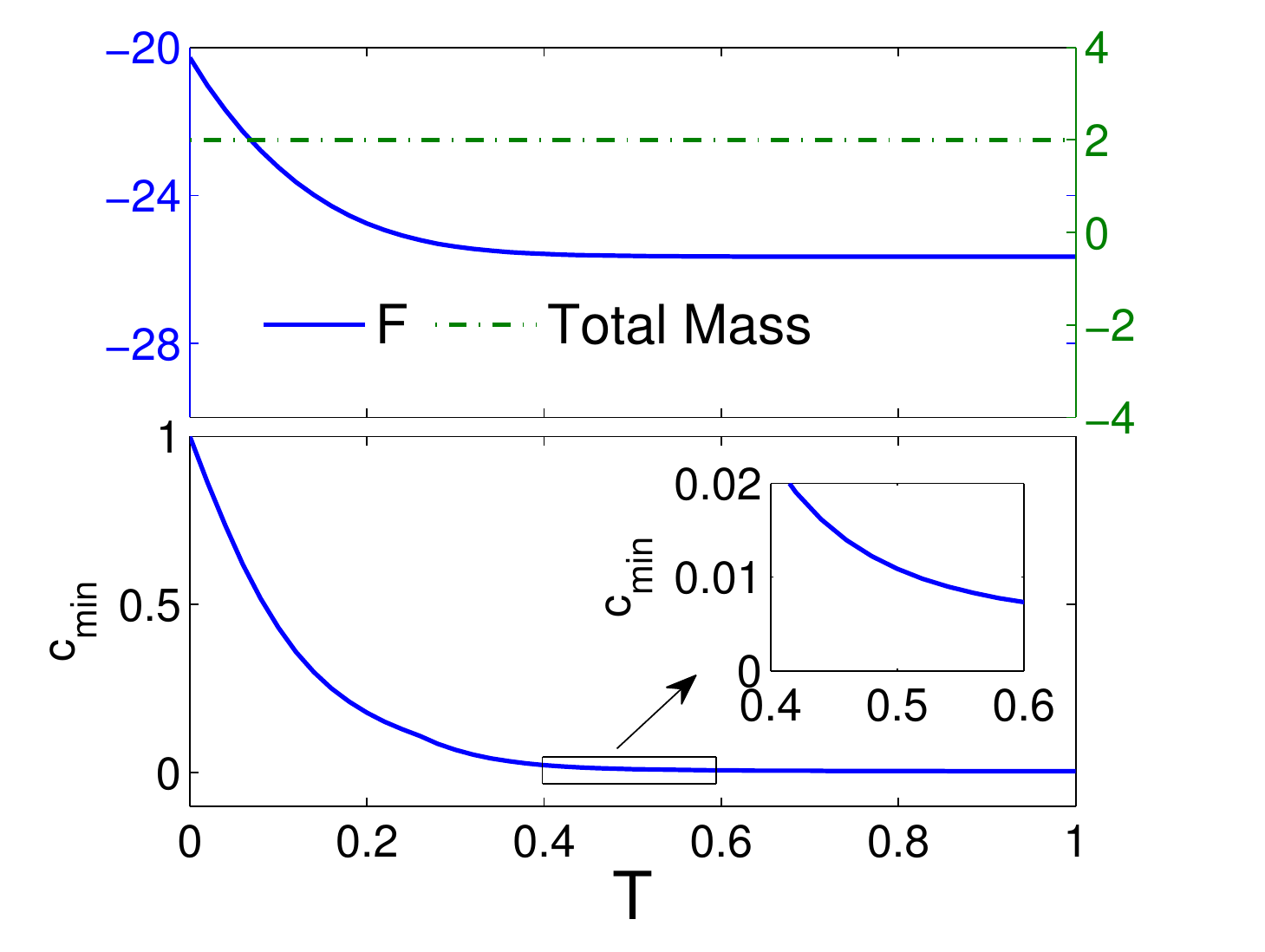}}
\caption{Upper: Time evolution of the free energy and total mass. Lower: Time evolution of the minimum concentration value of both cations and anions. The inset is a zoomed-in plot for a time interval $[0.4, 0.6].$}
\label{Eminc1d}
\end{figure}

We now check structure-preserving properties of the proposed scheme. As shown in the upper plot of Figure ~\ref{Eminc1d},  the total mass of ions represented by the dashed line stays constant for all the time. Also, the free energy decays monotonically, indicating that our numerical scheme is energy stable.   As ions are repelled by fixed charges of the same sign, the local ionic concentrations become very low. It is of physical interest to check positivity of numerical solutions of concentrations.   The lower plot of Figure~\ref{Eminc1d} displays the evolution of minimum values of concentrations of both cation and anion against time, and the inset presents a zoomed-in plot for a time interval $[0.4, 0.6]$. It is demonstrated that, although the concentrations can be very low due to electrostatic repulsion, the numerical solutions of ionic concentrations remain positive for all the time.

\subsection{Applications}
We now apply the PNPCH equations and the corresponding numerical method to study the spatial ionic arrangement and charge dynamics of concentrated electrolytes that have been widely used in various applications, such as electrochemical energy devices. Salient features of concentrated electrolytes include crowding and charge layering in electric double layers, multiple time scale dynamics, self-assembly of nanostructuring both in the bulk and electric double layers (EDLs)~\cite{GavishYochelis_JPCLett16, BGUYochelis_PRE17, GavishEladYochelis_JPCLett18}, etc.  The PNPCH equations, which combine the effect of phase separation with electrodiffusion, are used to investigate these features. 

\begin{figure}[htbp]
\centering
\includegraphics[scale=0.95]{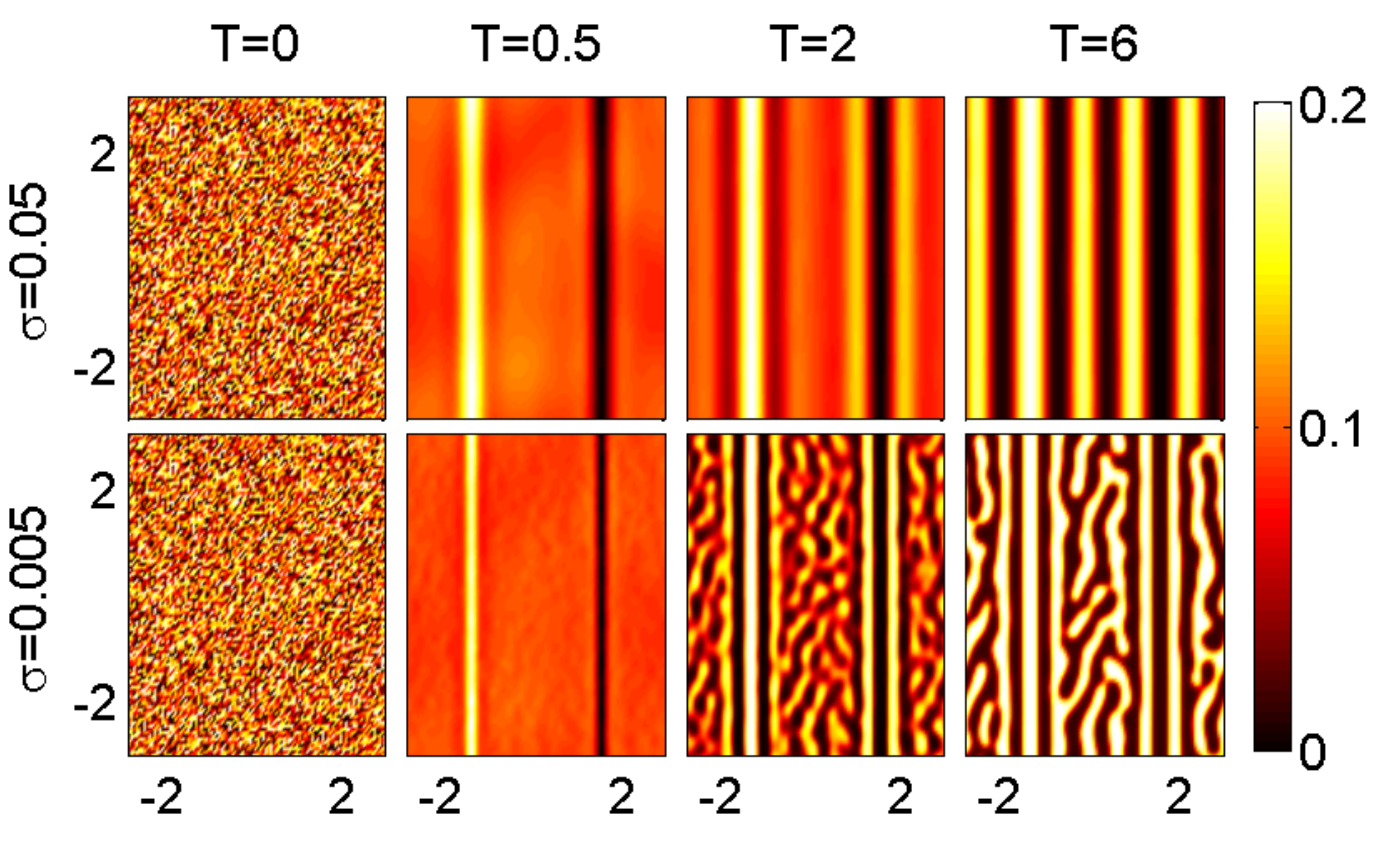}
\caption{Snapshots of the evolution of cation concentrations starting from a random initial data with $\sigma= 0.05$ and $\sigma= 0.005$.}
\label{2dTdiff}
\end{figure}

We consider the equations~\reff{PNPCHE}  on $\Omega=[-3,3]\times[-3,3]$ with periodic boundary conditions and a random initial data. The distribution of fixed charges are given by 
\[
\begin{aligned}
\rho^{f}(x,y)= -\frac12\chi_{\{x=-\frac{3}{2}\}}  + \frac12\chi_{\{x=\frac{3}{2}\}}, 
\end{aligned}
\]
where $\chi_{A}$ is the characteristic function over a set $A$. Such a distribution describes that negative and positive charges are distributed on the lines at $x=-\frac{3}{2}$ and $x=\frac{3}{2}$, respectively. We take gradient energy coefficients $\sigma^{1}=\sigma^{2}=\sigma$ and the steric interaction coefficient matrix $
G=\left(
\begin{matrix}
1 & g^{12}\\
g^{21}  & 1 \\
\end{matrix}
\right),$
where off-diagonal elements $g^{12}= g^{21} = 15$. 

Figure~\ref{2dTdiff} presents several snapshots of the evolution of cation concentrations with $\sigma= 0.05$ and $\sigma= 0.005$ at $T=0,$ $T=0.5,$ $T=2$, and $T=6$. Starting from a random initial distribution, the cations move quickly following the electrostatic potential mainly generated by the fixed charges. For $\sigma= 0.05$, we observe  at $T=2$ that, cations further crowdingly accumulate in the vicinity of negative fixed charges with the emergence of oscillations in diffuse layers of EDLs. This is reminiscent of the overscreening structure studied in the work~\cite{BSK:PRL:2011}. As time further evolves, the cation concentrations show periodic lamellar patterns not only in the EDLs but also in the bulk. For $\sigma= 0.005$, in contrast,  cations begin to develop labyrinthine type of structure in the bulk at $T=2$, but lamellar structures are still favored near fixed charges.  As time evolves, the patterns become clearer and clearer, leading to a totally different structure in comparison with that of $\sigma= 0.05$. The pronounced difference is ascribed to the gradient energy coefficients that penalize large concentration gradients. Smaller gradient energy coefficients allow more concentration oscillations.  Comparing snapshots at different time instants, we find that the electrostatic interactions dominate the ion migration in the early stage, and the effect of phase separation comes into play later in the development of patterns in the bulk. 

\begin{figure}[htbp]
\centering
{\includegraphics[scale=0.7]{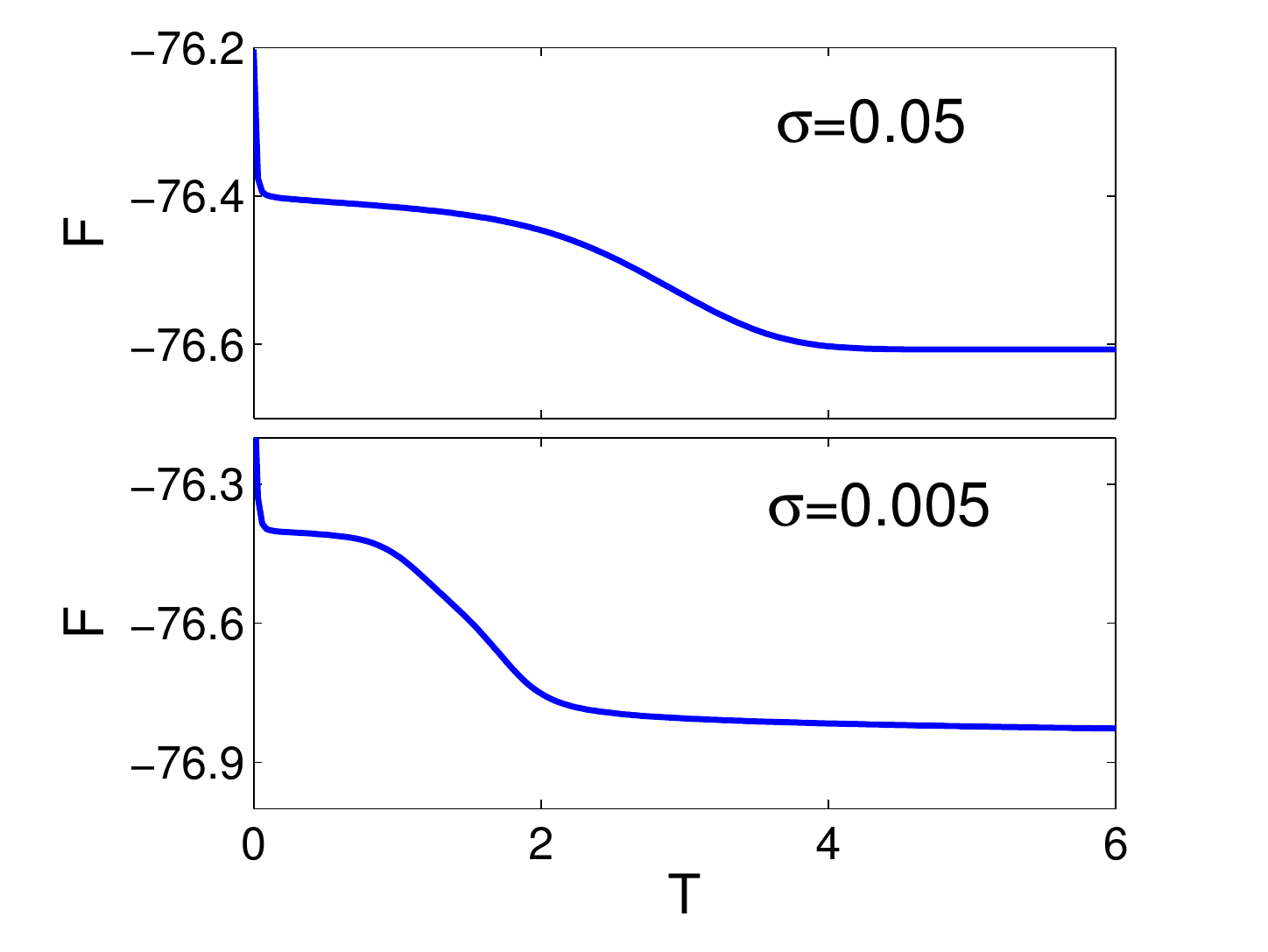}}
\caption{The evolution of free energy with $\sigma= 0.05$ and $\sigma= 0.005$.}
\label{Fmass2d}
\end{figure}

As shown in Figure~\ref{2dTdiff}, the mechanisms of phase separation and electrodiffusion take effects in different stages of the pattern formation, indicating the emergence of multiple time scale charge dynamics. To further understand the charge dynamics, we show in Figure~\ref{Fmass2d} the evolution of free energy of the system. A multi-phase free-energy dissipation can be clearly observed for both $\sigma= 0.05$ and $\sigma= 0.005$, reminiscent of metastability phenomena. In the first stage, the free energy decays sharply on a fast time scale, corresponding to the phase of electrodiffusion, and quickly reaches a metastable state characterized by a plateau in the free energy.  In the second stage, the free energy decays with a relatively longer time scale, corresponding to the formation of patterns in the bulk and EDLs, and reaches an equilibrium eventually.  Comparing the results with different $\sigma$ values, we also observe that, with larger gradient energy coefficients, the metastable state lasts longer and the time scale in the second stage of  free-energy dissipation is larger. Such multi-phase free-energy dissipation with metastability often takes  long time to reach an equilibrium.  Efficient numerical simulations of such dynamics require robust, energy stable numerical schemes that allow large time stepping. Our numerical results demonstrate that the proposed energy stable numerical scheme is capable of effectively capturing such multi-phase dynamics. 





\begin{figure}[htbp]
\centering
\includegraphics[scale=0.9]{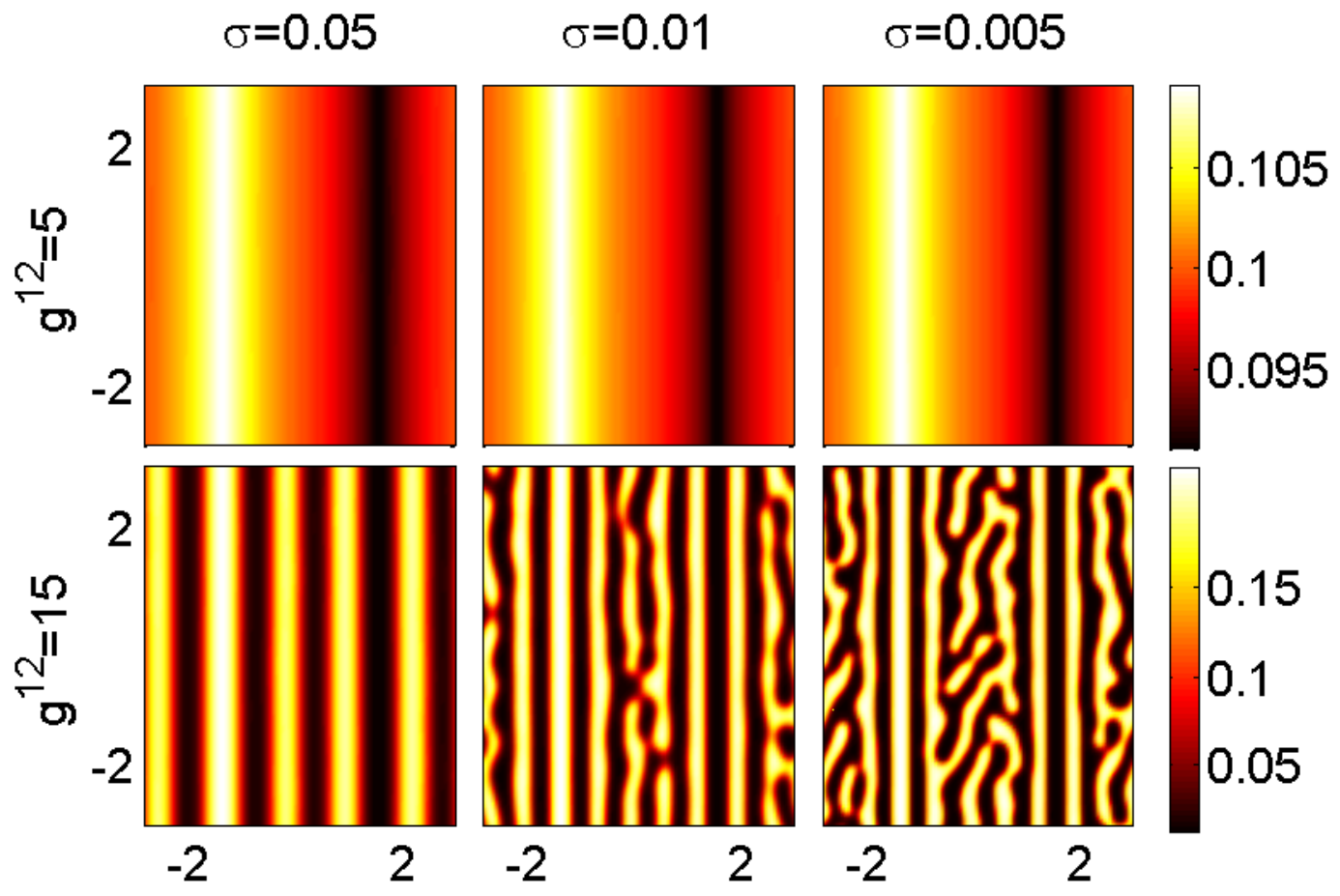}
\caption{Equilibrium states of cation concentrations with various combinations of $g^{12}$ and $\sigma$, starting from the same random initial data.}
\label{2dSteady}
\end{figure}
We next study the interplay between off-diagonal elements in the steric interaction coefficient matrix ($g^{12}$) and the gradient energy coefficient ($\sigma$), and their impact on the development of nanostructures in the equilibrium state. Note that the off-diagonal elements describe the cross interactions of short range.  Figure~\ref{2dSteady} plots equilibrium states of cation concentrations with various combinations of $g^{12}$ and $\sigma$, starting from the same random initial condition.  From the upper row plots, we find that, with a relatively small off-diagonal element $g^{12}=5$,  only EDLs in the vicinity of fixed charges develop in the equilibrium states. With a larger off-diagonal element $g^{12}=15$, we can see rich self-assembled patterns, such as lamellar stripes for a strong gradient energy coefficient $\sigma=0.05$ and labyrinthine patterns for a weak gradient energy coefficient $\sigma=0.005$. Mixed patterns with lamellar EDLs and labyrinthine structures in the bulk also present for an intermediate value of $\sigma=0.01$.   Comparison of two rows of plots reveals that strong cross interactions of short range are necessary for the formation of self-assembled nanostructures in the bulk. Comparison of three columns of plots indicates that more complex nanostructures develop with weaker concentration-gradient regularization.


\section{Concluding remarks}\label{s:Conclusions}
%

In this work, we have derived the PNPCH equations based on a free-energy functional that includes electrostatic free  energies, entropic contribution of ions, steric interactions, and concentration gradient energies. Numerical studies on the PNPCH equations are still missing, especially those concerning preservation of physical structures. We have proposed a novel energy stable, semi-implicit numerical scheme that guarantees mass conservation and positivity at the discrete level. Detailed analysis has revealed that the solution to the proposed nonlinear scheme corresponds to a unique minimizer of a convex functional over a closed, convex domain, establishing the existence and uniqueness of the solution. The positivity of numerical solutions has been rigorously proved via an argument on the singularity of the entropy terms at zero concentrations. Discrete free-energy dissipation has been established as well. Numerical tests on convergence rates have demonstrated that the proposed numerical scheme is first-order accurate in time and second-order accurate in space. Numerical simulations have also verified the capability of the numerical scheme in preserving the desired properties, e.g.,  mass conservation, positivity, and free energy dissipation. 

Moreover, we have applied the PNPCH equations and the proposed scheme to investigate charge dynamics and ionic arrangement, such as self-assembled nanostructures, in highly concentrated electrolytes.  In numerical simulations, we have found that there are multiple time relaxations with distinct time scales, and metastable states present in the relaxation to an equilibrium. Efficient simulations of such dynamics require robust, energy stable numerical schemes that allow large time stepping.  Our numerical results have demonstrated that the proposed numerical scheme is able to capture lamellar patterns and labyrinthine patterns in electric double layers and the bulk, as well as  multiple time scale dynamics with intermediate metastable states. In addition, we have probed the interplay between cross steric interactions and the concentration gradient regularization, and their profound influence on the pattern formation in the  equilibrium state.

\section*{Acknowledgments}
Y. Qian and S. Zhou were supported by National Natural Science Foundation of China (21773165) and National Key R\&D Program of China  (No.~2018YFB0204404). C. Wang was supported by NSF under DMS-1418689. 

 
\bibliographystyle{plain}
\bibliography{PNPCH}

\end{document}